\documentclass[11pt,reqno]{amsart}
\usepackage{amsmath, amssymb, amsfonts, amsthm,latexsym, hyperref,  mathdots, enumerate}
\usepackage{graphicx}
\usepackage{marginnote}

\def\R{\mathbb{R}}
\def\C{\mathbb{C}}
\def\P{\mathbb{P}}
\def\E{\mathbb{E}}
\def\Q{\mathbb{Q}}
\def\N{\mathbb{N}}
\def\Z{\mathbb{Z}}

\def\ALG{\mathbb{A}}

\renewcommand{\Re}{\mathrm{Re}}
\newcommand{\ga}{\mathbb{\alpha}}

\DeclareMathOperator{\Ber}{\mathrm{Ber}}

\newtheorem{observation}{Observation}[section]

\newtheorem{theorem}{Theorem}[section]
\newtheorem{proposition}[theorem]{Proposition}
\newtheorem{lemma}[theorem]{Lemma}
\newtheorem{claim}[theorem]{Claim}

\def\eps{\varepsilon}

\newcommand{\atom}{p_{\max}}

\newcommand{\kr}{\delta}
\newcommand{\al}{\alpha}
\newcommand{\dv}{d}
\newcommand{\ep}{\varepsilon}
\newcommand{\Pro}{\mathbb{P}}

\newcommand{\ber}{B}

\newcommand{\ind}{1{\hskip -2.5 pt}\hbox{I}}

\title{Double roots of random polynomials with integer coefficients}
\author{Ohad N. Feldheim}
\address{Ohad N. Feldheim\hfill\break
    Stanford University\\
    Department of Mathematics\\
    Stanford, California, 94305, USA.}
\thanks{Research of O.N.F. supported in part by the Institute for Mathematics and its Applications with funds provided by the National Science Foundation.}

\author{Arnab Sen}
\address{Arnab Sen\hfill\break
    University of Minnesota\\
    School of Mathematics\\
    206 Church St SE, Minneapolis, MN 55455, USA.}
\thanks{Research of A.S. supported in part by the American National Science Foundation grant, DMS-1406247}

\begin{document}

\maketitle

\begin{abstract}
We consider random polynomials whose coefficients are independent and
identically distributed on the integers. We prove that if the coefficient distribution has bounded support and its probability to take any particular value is at most
$\tfrac12$, then the probability of the polynomial to have a double root is dominated by the probability that either $0$, $1$, or $-1$ is a double root up to an error of $o(n^{-2})$. We also show that if the support of coefficient distribution excludes $0$ then the double root probability is $O(n^{-2})$. Our result generalizes a similar result of
Peled, Sen and Zeitouni \cite{PSZ} for Littlewood polynomials.
\end{abstract}

\section{Introduction}
%

Let $n\in\N$ and let $(\xi_j)_{0 \le j}$ be a sequence of independent and identically distributed
random variables taking values in $\Z$.
Define the random polynomial $P=P_n$ by
\begin{equation}\label{eq:poly_def}
  P(z):=\sum_{j=0}^n \xi_j z^j.
\end{equation}
In a previous paper, Peled, Sen and Zeitouni~\cite{PSZ} showed that if the random variables are supported on $\{-1, 0, +1\}$  with   $\max_{x\in\{-1,0,1\}} \P(\xi_0 = x) < \tfrac{1}{\sqrt{3}}$, then the probability of  $P$ to have a double root in the complex plane is same as having a double root at $0, \pm 1$ up to an error of $o(n^{-2})$.  In this paper, we extend the result for more general integer-valued random variables.  Our main result is the following.
\begin{theorem}\label{thm:main}
Suppose the coefficient distribution satisfies the following conditions.
\begin{equation}\label{eq:bd_support}
\text{ There exists constant } M \ge 1 \text{ such that } \P( |\xi_0| \le M) =1.
\end{equation}
\begin{equation}\label{eq:coefficient_condition}
\max_{x\in\Z} \P\Big(\xi_0 = x\Big) \le \frac{1}{2}.
\end{equation}
Then we have
  \begin{equation}\label{eq:thm_conclusion}
    \Pro\Big(P\text{ has a double root}\Big) = \Pro\Big(P\text{ has a double root at either } 0, -1 \text{ or } 1 \Big)   +  o(n^{-2})\quad\text{as $n\to\infty$.}
  \end{equation}
  Moreover, if $\P(\xi_0=0)=0$, then  $\Pro\Big(P\text{ has a double root}\Big)  = O(n^{-2}).$
\end{theorem}

We make a few remarks  about the above theorem.
\begin{enumerate}
\item When $\P(\xi_0=0)=0$,   the upper bound in Theorem~\ref{thm:main} is sharp. When $\xi_i$'s are i.i.d.\ $\pm 1$ symmetric Bernoulli and $(n+1)$ is divisible by $4$, then  it was shown in \cite{PSZ} that the probability of having a double root is  $\Theta(n^{-2})$.

\item We can have a  better error bound if we allow the possibility of having double roots other low-degree roots of unity. More precisely,
our proof can be modified to show that for any fixed $d \ge 1$,
\begin{align*}
 & \hspace{2cm} \Pro\Big(P\text{ has a double root}\Big) \\
 &= \Pro\Big(P\text{ has a double root at } 0 \text{ or some roots of unity of degree at most } d \Big)   +  o(n^{-2d}).
 \end{align*}
\item The bounded support condition~\eqref{eq:bd_support} can be weaken with minor modifications of our arguments. We did not pursue that here for sake of simplicity. On the other hand, we do not know how to relax condition~\eqref{eq:coefficient_condition} on the size of the maximum of atom and it seems that the current bound $\tfrac12$ is a limitation of our proof. In fact, we believe that both conditions are unnecessary and that the result \eqref{eq:thm_conclusion} should hold for any non-degenerate integer-valued coefficient distribution.

\item Even though some parts of our proof closely follow the lines of arguments from the paper of Peled, Sen and Zeitouni~\cite{PSZ}, extending the result to general integer-valued coefficients, however,  poses a few significant challenges. For example, to handle high-degree double roots, we need a key anti-concentration estimate for $P(\pm 2)$  given in the form of Theorem~\ref{thm: main}. When the coefficients are $\pm 1$-valued, the map $(a_0, \ldots, a_n) \mapsto \sum_{i=1} a_i 2^i : \{-1, 1\}^{n+1} \to \Z$ is one-to-one, which immediately implies the bound that
 \[  \P( P(\pm 2) = m) \le \big(\max_{x = \pm 1} \P( \xi_0 = x) \big)^{n+1}.\]
 The paper \cite{PSZ} made use of the above simple observation. But for more general integer-valued coefficients, we lose such one-to-one property, which makes  proving Theorem~\ref{thm: main} nontrivial. One consequence of this difficulty  is that the result here requires the maximal atom of the coefficient distribution to be at most $\tfrac12$ while in \cite{PSZ} atoms up to $\tfrac{1}{\sqrt{3}}$ could be handled.

Moreover, the argument used in \cite{PSZ} to deal with low degree roots does not carry over either.  In \cite{PSZ}, $P$ was always a monic polynomial and hence its roots were algebraic integers. For algebraic integers, one can use some partial result (see, e.g., Dobrowolski
 \cite{D79})
 on  Lehmer's conjecture to show that any non-zero algebraic integer is either a root of unity or has a conjugate which is a bit far (depending on its degree) away from unit circle. In low degree case, \cite{PSZ} made use of this dichotomy of algebraic integers. In contrast, in our case we also have to deal with non-monic $P$, so its roots are algebraic numbers in general. Such dichotomy is not available for algebraic numbers. For example, there are algebraic numbers which are not a root of unity and all of its conjugates lie on the unit circle. This requires new methods for handling such roots which are given in sections~\ref{sec:roots_near_unit} and~\ref{sec:tiny_deg}.

\end{enumerate}

A key instrument in the proof of the theorem, which may be of independent interest,  is the following anti-concentration bound.

\begin{theorem}\label{thm: main}
Under condition \eqref{eq:coefficient_condition} there exists $\ep>0$ such that for all $n \in\N$ large enough we have
\[
\max_{m \in \Z} \Pro\Big(P(\pm2)=m\Big)\le 2^{-n(\frac12+\ep)}.
\]
\end{theorem}
%

We proceed as follows. In Section~\ref{subs:overview} we provide some notation and reduce Theorem~\ref{thm:main} to several key lemmata. In Section~\ref{sec:anticon} we prove Theorem~\ref{thm: main}. Each subsequent section is then dedicated to the proof of one of the key lemmata stated in Section~\ref{subs:overview}.

\subsection{Proof Overview.}\label{subs:overview}
%

{\bf Preliminaries.} Recall that a real number $\alpha$ is called \emph{algebraic} if it is a root of a
polynomial with rational coefficients. Let $\ALG$ denote the set of algebraic numbers. The \emph{minimal polynomial} of $\alpha\in\ALG$
is the unique least degree monic polynomial in $\Q[X]$ with a root at $\alpha$.
The \emph{algebraic degree} of $\alpha$ is the degree of the minimal polynomial of $\alpha$, which we denote by $\deg(\alpha)$.
A real number $\alpha$ is said to be an \emph{algebraic integer} if all the coefficients of its minimal polynomial are integers.

We define $\Lambda(\alpha)$, the \emph{house} of $\alpha$, by
$$\Lambda(\alpha)=\max_{j\in\{1,\dots,\deg(\alpha)\}}|\alpha_j|,$$
Where $\ \alpha_1 = \alpha,\dots, \alpha_{\deg(\alpha)}$ are the conjugates of $\alpha$, i.e., the roots of the minimal polynomial of $\alpha$.

We further define the \emph{associated minimal polynomial} of $\alpha$ in $\Z[x]$ to be the unique polynomial in $\Z[x]$ of degree $\deg(\alpha)$ with a root at $\alpha$, whose leading coefficient is positive and whose coefficients are coprime.
\vspace{5pt}

\noindent{\bf Main lemmata.} The proof of Theorem~\ref{thm:main} breaks into several cases. In what follows in this subsection,  we let $P$ be as in Theorem~\ref{thm:main} with coefficient  distribution satisfying \eqref{eq:bd_support} and \eqref{eq:coefficient_condition}.

We first consider the probability of having a double root of algebraic degree and prove the following result.
\begin{lemma}[high degree]\label{lem:high_degree}
Given any $B>0$, there exist a constant $C_0>0$ such that
\[\P(P\text{ has a double root $\alpha$ with $\deg(\alpha)\ge
    C_0\log n$})= O(n^{-B}).\]
\end{lemma}
The proof of Lemma~\ref{lem:high_degree} follows the line of arguments given in \cite{PSZ}, which, in turn, was based on idea that appeared in a work of Filaseta and Konyagin \cite{FK96}. However, several modifications are needed when dealing with general integer-valued coefficients. Most crucially, we need a new anti-concentration bound (Theorem~\ref{thm: main}) that consumes the bulk of our effort. Let us point out here that Theorem~\ref{thm: main}  is the only place where Assumption~\eqref{eq:coefficient_condition} is crucially used.

By virtue of Lemma~\ref{lem:high_degree}, we now have to deal with potential double roots with low algebraic degree, more precisely, with degree at most $C_0 \log n$.  In the next lemma we show that the probability that $P$ has a root at an algebraic numbers of low degree such that one of its conjugates lying at a distance of at least $\Omega( (\log n)^{-1})$ from the unit circle is negligible.
\begin{lemma}[low degree roots far away from the unit circle]\label{lem: med deg large house}
For every $B>0$ and $C_0>0$, there exists $C_1>0$ such that
$$\Pro\Big(P\text{ has a root } \alpha\ :\
\deg(\alpha)\le C_0\log n\text{ and }
\Lambda(\alpha)>1+\tfrac{C_1}{\log n}\Big)= O(e^{ - \frac{B n}{\log n }}).$$
\end{lemma}
For the proof, we use a simple sparsification of $P$ to bound the root probability for each fixed low-degree algebraic number lying far away from the unit circle and then employ a rather crude union bound.
After Lemma~\ref{lem: med deg large house}, we next deal with the low degree double roots with small house (i.e.\  all of their conjugates lying close to the unit circle).  We break this into two cases. First we consider the case when the degree of the root is at least $5$ and we show that
\begin{lemma}[low degree roots close to the unit circle]\label{lem: med deg small house}
For every $C_0>0$  and $C_1>0$, we have
$$\Pro\left(P\text{ has a root } \alpha\ :\ 4<\deg(\alpha)\le C_0\log n\text{ and }\Lambda(\alpha)\le1+\tfrac{C_1}{\log n}\right)=o(n^{-2}).$$
\end{lemma}
From a standard application of inverse Littlewood-Offord type results, it follows that for any fixed algebraic number $\alpha$ of degree at least $5$, $\P(P(\alpha) = 0) = O_\eps(n^{-5/2 +\eps}),$ for any $\eps>0$. This is shown in Lemma~\ref{lem:med_deg}. More importantly,  to prove Lemma~\ref{lem: med deg small house}, we need to count the number of  algebraic numbers $\alpha$ such that $deg(\alpha)\le C_0\log n\text{ and }\Lambda(\alpha)\le1+\tfrac{C_1}{\log n}$. Towards this direction, we show in Lemma~\ref{lem:poly_count} that they are at most $o(n^{\eps})$ in number for any $\eps>0$. The counting estimate makes heavy use of a result of Dubickas \cite{Du}.

Finally, the next lemma  takes care of the potential double roots that have degree at most $4$ (excluding $0, \pm 1$) and have small house.
\begin{lemma}[roots with degree at most $4$] \label{lem: small deg small house}
For every $C_0>0$  and $C_1>0$, we have
$$\Pro\left(P\text{ has a double root } \alpha\neq0,\pm1 :\ \deg(\alpha)\le 4\text{ and }\Lambda(\alpha)\le1+\tfrac{C_1}{\log n}\right)=o(n^{-2})$$
\end{lemma}
It is not hard to see that if  $\alpha$ is a root of $P$ for large enough $n$ satisfying the conditions that $\deg(\alpha) = O(1)$ and $\Lambda(\alpha) = o(1)$, then  it must be a unimodular root, i.e., all of the conjugates of $\alpha$ must lie on the unit circle. Now if $\alpha$ is a root of unity, we closely follow  \cite{PSZ} to bound the probability of having a double root $\alpha$ which involves an
application of an anti-concentration bound due to S\'ark\"ozi and Szemer\'edi \cite{SS65}.  However, when $\alpha$ is unimodular but not a root of unity, we need a new argument to bound the probability of having a double root at $\alpha$. In fact,  in Lemma~\ref{lem:unimodular_but_not_root_unity} we show that $\P(P(\alpha) = 0) = O(n^{-5/2})$. The argument relies on a powerful anti-concentration bound by Hal\'asz \cite{H77}.

Clearly, the first assertion of Theorem~\ref{thm:main} is an immediate consequence of lemmata~\ref{lem:high_degree}, \ref{lem: med deg large house}, \ref{lem: med deg small house},~\ref{lem: small deg small house}.  To prove the second assertion of Theorem~\ref{thm:main} , note that since $P(\xi_0=0)=0$, with probability one, $P$ can not have a root at $0$. So, we need to show that
\[ \P( P \text{ has a double root at } \pm 1 ) = O(n^{-2}).\]
The above bound follows from an application of optimal inverse Littlewood-Offord  theorem \cite[Theorem 2.5]{NV11}.
For details, see Lemma~A.5 in \cite{DNV} where the same has been proved under the assumption that $\xi_0$ has bounded $(2+\eps)$ moment. This completes the proof of Theorem~\ref{thm:main}.

\section{Anti-concentration of $P(\pm2)$}\label{sec:anticon}

In this section we prove Theorem~\ref{thm: main}. As an important first step, we will find a very useful a characterization of integer-valued measures with max-atom bounded by $\tfrac12$ in terms of mixture of two-point distributions.


\subsection{Bernoulli mixture}
A probability measure $\mu$ is said to be a \emph{(unbiased) Bernoulli measure} if   $\mu = \tfrac12 \delta_a + \tfrac12 \delta_b,$  where $a \ne b \in \Z$ and $\kr_x$ is the Dirac measure at $x$.  A countable mixture of  unbiased Bernoulli measures is simply said to be a \emph{Bernoulli mixture}.  In other words, a probability measure $\mu$ is a Bernoulli mixture if
it can be written as  \[ \mu=\frac12\sum_{i=1}^\infty t_i(\kr_{a_i}+\kr_{b_i}) \]
where $t_i\ge 0$ satisfy $\sum_i t_i =1 $ and  $a_i \ne b_i\in \Z$ for each $i$.

Note that if the distribution of a random variable $\xi$ is a Bernoulli mixture, then there exist a random vector $(I, \Delta)$ on $\mathbb{Z} \times \mathbb{N}$, such that
\begin{equation}\label{eq:bm_rep}
\xi \stackrel{d}{=} I + B \Delta,
\end{equation}
where $B$ is a $\Ber(\tfrac12)$   random variable,  independent from both $I$ and $\Delta$. With a slight abuse of notation, we will also call such a random variable $\xi$ a Bernoulli mixture.



The following proposition gives a useful characterization for Bernoulli mixtures.
\begin{proposition}[Bernoulli mixture]\label{propos:mixture-reduction}
A integer-valued random variable $\xi$  is a Bernoulli mixture if and only if it  satisfies $\max_{x\in\Z}\Pro\big(\xi=x\big)\le 1/2$.
\end{proposition}
Clearly, the necessary part of Proposition~\ref{propos:mixture-reduction} is trivial. Most of the reminder of Section~\ref{sec:anticon}
 is dedicated to proving the sufficient part.


Let $\mu$ be a non-negative positive finite measure on $\Z$. It induces a unique total order $(\pi_i^\mu)_{i\in\N}$ on $\Z$ such that $w^\mu_i:=\mu(\pi_i^\mu)$ are monotone non-increasing (i.e., $w^\mu_i\ge w^\mu_j$ if $i<j$) and $\pi_i^\mu < \pi_j^\mu$  if $w^\mu_i =  w^\mu_j$. Then $\mu$ can be expressed as
$$\mu = \sum_{i\in\Z}w^\mu_i\kr_{\pi_i^\mu}.$$
We write $\mathcal M$ for the collection of non-negative finite measures $\mu$
on the integers (including the null measure), which satisfy $w^{\mu}_1 \le \mu(\Z)/2$. Also, for any non-null finite measure  $\mu$ on $\Z$, we denote by $\bar \mu$ the normalized probability measure $\bar \mu (\cdot)  := \mu( \cdot)/ \mu (\Z)$.

To prove Proposition~\ref{propos:mixture-reduction} we use the following couple of lemmata.
\begin{lemma}\label{lem:three-support case}
If  $\mu\in \mathcal M$ is non-null and $\mu$ is supported on at most $3$ integers, then $\bar \mu$
is a mixture of at most $3$ Bernoulli measures.
\end{lemma}
\begin{proof}
We write
$$\bar \mu = w_1\kr_{\pi_1}+w_2\kr_{\pi_2}+w_3\kr_{\pi_3}$$ where $w_1\ge w_2\ge w_3$ and $\sum_{i=1}^3 w_i =1$.
We then give the explicit decomposition:
\begin{align*}
\mu =   (w_1 + w_2 -w_3) \Big( \tfrac12  \kr_{\pi_1} + \tfrac12  \kr_{\pi_2} \Big) + (w_1 + w_3 -w_2) \Big( \tfrac12  \kr_{\pi_1} + \tfrac12  \kr_{\pi_3} \Big) + (w_2 + w_3 -w_1) \Big( \tfrac12  \kr_{\pi_2} + \tfrac12  \kr_{\pi_3} \Big).
\end{align*}

It is now straightforward to check that each of the weights is non-negative and that equality indeed holds.
\end{proof}

\begin{lemma}\label{lem:extract-one case}
Let $k\ge 4$ be an integer. Every $\mu\in \mathcal M$ can be written as
$$\mu = \nu + \beta,$$
where either $\beta$ is either the null measure or $\bar \beta$ is a Bernoulli measure, and $\nu\in\mathcal M$ and satisfies $\nu(\pi^\mu_k)=0$.
\end{lemma}

\begin{proof}
Set $\beta :=w^\mu_k \kr_{\pi^\mu_1}+w^\mu_k \kr_{\pi^\mu_k}$ and $\nu := \mu - \beta$. It only remains to check that $\nu\in\mathcal M$, that is, the fact that $w_1^\nu\le \nu(\Z)/2$. To see this, observe that $\pi^\nu_1\in\{\pi^\mu_1,\pi^\mu_2\}$.
If $\pi^\nu_1=\pi^\mu_1$, then, since $\mu\in\mathcal M$, we have
$$w_1^\nu=w_1^\mu-w_k^\mu \le \frac{\mu(\Z)}2-w_k^\mu = \frac{\nu(\Z)}2.$$
On the other hand, if $\pi^\nu_1=\pi^\mu_2$, we get that
$$w_1^\nu = w_2^\mu \le  \frac{w_1^\mu + w_2^\mu+ w_3^\mu-w_k^\mu}2 = \frac{\nu(\pi_1^\mu)+\nu(\pi_2^\mu)+\nu(\pi_3^\mu)}2 \le \frac{\nu(\Z)}2.$$
The lemma follows.
\end{proof}

\begin{proof}[Proof of Proposition~\ref{propos:mixture-reduction}]

Write $\mu_1$ for the distribution of $\xi$. Define a decreasing sequence of finite measures
  $(\mu_i)_{i\in\N}$ on $\Z$  inductively as follows.
  Suppose $\mu_i$ has already been defined and $\mu_i \in \mathcal M$. An application of Lemma~\ref{lem:extract-one case} to $\mu_i$ with $k=4$ yields the decomposition
  $\mu_i  = \beta_i + \mu_{i+1}$ with $\mu_{i+1}(\pi_4^{\mu_i}) = 0$ where $\beta_i$ is either the null measure or $\bar \beta_i$ is a Bernoulli measure and
  $\mu_{i+1} \in \mathcal M$. This defines the measure $\mu_{i+1}$.
%
Since $(\mu_i)_{i\in\N}$ is a decreasing sequence of finite measures,   it has a limiting measure (possibly null) which we denote by $\mu_{\infty}$.
Thus we write
$$\mu=\sum_{i\in\N}\beta_i+\mu_{\infty}.$$

All that remains in order to prove the proposition is to show that $\mu_{\infty}$ is supported on at most $3$ integers, and then apply Lemma~\ref{lem:three-support case}.

To that end, assume, if possible that, there exists four distinct integers $a_1, a_2, a_3, a_4$ such that $\mu_\infty(a_i) >0$ for all $i$. Set $c := \min \{ \mu_\infty(a_i): 1\le i \le 4  \}>0$. For each $i \in \N$,  define the set
\[ L_i := \{ x\in \Z : \mu_i(x) \ge c\}.\]
Since $\mu_i \downarrow \mu_\infty$, $L_i \supseteq L_{i+1}$ and $a_1,\dots,a_4 \in L_i$ for each $i$. Thus $ \pi_4^{\mu_i} \in L_i$ and hence, by the definition of the measure $\mu_{i+1}$, we have  $L_i \subseteq L_{i+1} \setminus \{\pi_4^{\mu_i}\}$.  This implies that $|L_{i+1}| < |L_{i}|$ for each $i$. Since $|L_1| < \infty$, this contradicts the fact that $|L_i| \ge 4$ for each $i$. Hence,  $\mu_{\infty}$ is supported on at most $3$ integers.
\end{proof}

%
%
%

Using Proposition~\ref{propos:mixture-reduction} we may reduce Theorem~\ref{thm: main} to the following proposition.

\begin{proposition}\label{propos: main2}
Let $(X_i)_{ 1 \le  i  \le n}$ be i.i.d.\  random variables whose distribution is
a Bernoulli mixture.
Then there exists $\ep>0$ such that for $ n \in\N$ large enough and every sign sequence $ (\sigma_i)_{1 \le i \le n}$ with $\sigma_i=\pm1$, the following holds.
\[
\max_{m\in\Z}\Pro\big(\sum_{i=1}^n 2^i  \sigma_i X_i=m\big)\le 2^{-n(\frac12+\ep)}.
\]
\end{proposition}

\subsection{Proof of Proposition~\ref{propos: main2}}
In this section we prove Proposition~\ref{propos: main2}. Throughout the section we fix a sign sequence  $ (\sigma_i)_{1 \le i \le n}$ with $\sigma_i=\pm1$.
In the course of the proof we shall make several claims whose proofs are given in sections~\ref{sec: claim in thm} and~\ref{sec: lemma}.

 For $k\in\Z$ we write $L(k)$ for the leading power of $2$ in the factorization of $k$, i.e., $L(k) = \max \{ l \in \Z_+ : 2^l  \text{ divides } k \}$.
Since $(X_i)_{1 \le i \le n} $ are i.i.d.\ Bernoulli mixtures, following the representation \eqref{eq:bm_rep}, we can express $X_i$ as
$$X_i = I_i + {\ber}_i \Delta_i,$$
where $(I_i,\Delta_i)_{1 \le i \le n}$ are i.i.d.\ random vectors in $\Z \times \N$ and
 $(\ber_i)_{1 \le i \le n}$ are i.i.d.\ $\Ber(\tfrac12)$ random variables, which are independent from $(I_i,\Delta_i)_{1 \le i \le n}$.
Define
\[\atom:=\max_{m \in\Z}\Pro\big(\sum_{i=1}^n 2^i \sigma_i X_i=m\big).\]
We obtain an upper bound on $\atom$ using the following claim, whose proof we delay to Section~\ref{sec: claim in thm}.

\begin{claim}\label{cl: coins}
Let  $(\ber_i)_{1 \le i \le n}$ be i.i.d.\ $\Ber(\tfrac12)$ random variables  and let
$(b_i)_{ 1 \le i \le n} $ and $(d_i)_{ 1 \le i \le n} $ be any two sequences of integers.
Then
\[\max_{m\in\Z}\Pro\Big(\sum_{i=1}^n b_i+\dv_i B_i=m\Big)\le2^{-|\{L(\dv_i)\ : \ 1 \le i \le n \}|}.\]
\end{claim}
Applying
Claim~\ref{cl: coins} we have,
\begin{equation*}
\atom
\le \sum_{s=1}^n \Pro\Big( \big |\{L(2^i \sigma_i\Delta_i)\ :\  1 \le i \le n \} \big |=s\Big)2^{-s}.
\end{equation*}

Observing that $L(2^i \sigma_i\Delta_i)=L(2^i\Delta_i)$ for all $i$,
it would suffice to show that for large enough $n$ we have,
\begin{equation*}
\sum_{s=1}^n \Pro\Big( \big |\{L(2^i \Delta_i)\ :\  1 \le i \le n \} \big |=s\Big)2^{-s} \le  2^{-n(1/2+\ep)}.
\end{equation*}
Here and in the rest of the proof we let $\ep$ be a small positive constant, chosen to satisfy various constraints which are specified
along the proof.

Taking $w_i:= L(\Delta_i)$ for $1 \le i \le n$, and $W :=|\{ i +w_i \ :\ 1 \le i \le n\}|$, we may rewrite the above inequality as
\begin{equation}\label{eq: enough1}
\sum_{s=1}^n  \Pro\big(W=s\big)2^{-s} \le 2^{-n(1/2+\ep)}.
\end{equation}

By rewriting the LHS of \eqref{eq: enough1} as
\[\sum_{1\le s< n(1/2+\ep)} \hspace{-15pt}\Pro\big(W=s\big)2^{-s}+
\hspace{-7pt}\sum_{n(1/2+\ep)\le s \le n} \hspace{-13pt}\Pro\big(W=s)2^{-s},\]
we observe that
\begin{equation}\label{eq: enough2}
\sum_{s=1}^n  \Pro\big(W=s\big)2^{-s}<  n \left(\max_{\al\in(0,1/2+\ep)}\Pro\big(W=\al n\big)2^{-\al n}+2^{-n (1/2+\ep)}\right).
\end{equation}
Plugging \eqref{eq: enough2} into \eqref{eq: enough1} we get that it would be
enough to show the existence of $\ep >0$ such that for large enough
$n$,
\begin{equation*}
\max_{\al\in(0,1/2+\ep)}\Pro\big(W=\al n \big)2^{-\al n } \le 2^{-n(1/2+\ep)}.
\end{equation*}
Multiplying both sides by $2^{n/2}$ it reduces to showing that for $n$ sufficiently large,
\begin{equation}\label{eq: enough4}
\max_{\al\in(0,1/2+\ep)}\Pro\big(W=\al  n \big)2^{n (1/2-\al)} \le 2^{-\ep n}.
\end{equation}

In order to show  \eqref{eq: enough4}, we use the following lemma.
\begin{lemma}\label{lem: crucial bound imp}
Let $(w_i)_{1 \le i \le n}$ be i.i.d.\ non-negative integer-valued random variables,  and define $W=|\{i +w_i\ :\  1 \le i \le n\}|$.
Then the following holds.
\begin{enumerate}[(a)]
\item \label{item: lem1} For any $\al\in(0,1)$ with $\al n \in \N$, we have
\begin{equation*}
\Pro\Big(W= \al n \Big) \le \binom{n}{\al n}\al^n\le \left(\frac{\al}{1-\al}\right)^{n(1-\al)}.
\end{equation*}

\item\label{item: lem2} Furthermore, there exists $\delta',\ep'>0$ depending on $\al$ and the law of $w_1$  such that if $\al\in(1/2-\delta',1/2+\delta')$, then
\begin{equation*}
\Pro\Big(W= \al n \Big) \le e^{-\ep' n} \left(\frac{\al}{1-\al}\right)^{n(1-\al)}.
\end{equation*}
\end{enumerate}

\end{lemma}
Proving Lemma~\ref{lem: crucial bound imp} is the main technical step in the
proof of Proposition~\ref{propos: main2}, and we devote Section~\ref{sec: lemma} to
its proof.

The following claim, whose proof we delay to Section~\ref{sec: claim in thm}, captures two technical properties of the bound obtained in
Lemma~\ref{lem: crucial bound imp}.
\begin{claim}\label{cl: monotone}
For $n \in\N$, we define a function $f_n: (0, 1) \to \R_+$ as
\[
f_n(\al):=\left(\frac{\al}{1-\al}\right)^{n(1-\al)}2^{n(1/2-\al)}.
\]
Then the following hold. \begin{enumerate}[(a)]
\item \label{item: cl1}There exists $c_0>1/2$ such that $f_n(\al)$ is
    strictly monotone
    increasing in $(0,c_0)$.
\item \label{item: cl2}  Let $c_0$ be as in part (a). Then for any $c>0$ there exists $0<\delta<c_0-\frac 1 2$
    such that
    $f_n(\tfrac 1 2 +\delta)< 2^{cn}.$

\end{enumerate}
\end{claim}

Finally we are fully equipped to demonstrate the existence of $\ep>0$ such that \eqref{eq: enough4}
holds. Let $\delta',\ep'$ be as in part~\eqref{item: lem2} of Lemma~\ref{lem: crucial bound imp} and let $c_0$ be as in part (a) of Claim~\ref{cl: monotone}.  By part~\eqref{item: cl2}
of Claim~\ref{cl: monotone}, applied to $c=\ep'/2$, we get that there exists $\delta \in (0, c_0 -\tfrac12)$ such that $f_n\big(\tfrac 1 2 +\delta\big)<2^{\frac{\ep'n}2}$.
We take  $\eps :=\min(c_0-\tfrac12 ,\delta', \delta, \tfrac{\ep'}{2})$. We are thus left with verifying \eqref{eq: enough4}.

 Applying
Part~\eqref{item: lem1} of Lemma~\ref{lem: crucial bound imp} and
part~\eqref{item: cl1} of Claim~\ref{cl: monotone}, we get that,
\begin{align}\label{eq: end1}
I_1:=&\max_{\al\in(0,1/2-\ep]}\Pro\big(W=\al  n \big)2^{n (1/2-\al)} \le \max_{\al\in(0,1/2-\ep]}f_n(\al)=f_n\big(\tfrac 1 2-\ep \big) \notag\\
=&2^{ n \left((\frac12+\ep )\log_2\left(\frac{\frac12-\ep}{\frac12+\ep}\right)+\ep\right)}<2^{ n \left((\frac12+\ep)\log_2\left(1-\frac{2\ep}{\frac12+\ep}\right)+\ep\right)}<2^{-\left(\frac{2}{\log2}-1\right)\ep n}<2^{-\ep n},
\end{align}
using the inequality $\log_2(1-x)<-\frac{x}{\log2}$ for $x>0$.
From part~\eqref{item: lem2} of Lemma~\ref{lem: crucial bound imp} we obtain
\begin{align*}\label{eq: end2}
I_2 := & \max_{\al\in(1/2-\ep,1/2+\ep)}\Pro\big(W_n=\al n\big)2^{n(1/2-\al)} \le 2^{-\ep' n}\max_{\al\in(1/2-\ep,1/2+\ep)}f_n(\al)=2^{-\ep' n}f_n\big(\tfrac 1 2 +\ep\big) \notag \\
&\le 2^{-\ep' n}f_n\big(\tfrac 1 2 +\delta \big) \le 2^{-\ep' n} 2^{\frac{\ep'n}2}  = 2^{ - \frac{\ep'n}2} \le 2^{-\eps n}.
\end{align*}
 Therefore $\max ( I_1, I_2 ) <2^{-\ep n}$ and we obtain \eqref{eq: enough4}, as required. \qed

\vspace{15pt}

We remark that if our interest was limited to obtaining the theorem for the case
$\ep=0$, it would have been possible to use only the first part of
Lemma~\ref{lem: crucial bound imp}, which is, as will become evident, easier to obtain.

\subsection{Proofs of the claims}\label{sec: claim in thm}
This subsection consists of the proofs of the two claims used in the proof of Theorem~\ref{thm: main}.

\begin{proof}[Proof of Claim~\ref{cl: coins}]
Let $k:= |\{L(\dv_i)\ :\  1 \le i \le n\}|$. Let us assume, without loss of generality, that the values of $L(\dv_1), L(\dv_2), \ldots, L(\dv_k)$ are distinct and moreover, $L(\dv_1) < L(\dv_2)<  \cdots < L(\dv_k)$. Now, by conditioning on the random variables $B_{k+1}, B_{k+2}, \ldots, B_n$, we have
\[ \max_{m\in\Z}\Pro\Big(\sum_{i=1}^n b_i+\dv_i B_i=m\Big)  = \max_{m\in\Z}\Pro\Big(\sum_{i=1}^n \dv_i B_i=m\Big) \le \max_{m\in\Z}\Pro\Big(\sum_{i=1}^k \dv_i B_i=m\Big). \]
%
%
%
It would now suffice to show that for all $m\in\Z$ we have
$$\Pro\big(\sum_{i=1}^k  \dv_{i} \ber_{i}=m\big)\le 2^{-k}.$$
To see this, it would be enough to show that $\sum_{i=1}^k  \dv_{i} \ber_{i}$ takes distinct values for every choice of values of
$(\ber_{i})_{1 \le i \le k}$ in $\{0,1\}^k$.  Indeed, let
$(a_i)_{1 \le i \le k}$ and $(a'_i)_{1 \le i \le k}$ be two distinct vectors
of $\{0,1\}^k$, and let $r = \min \{ i \in \N: a_i \ne a'_i\}$. By definition,
$$\sum_{i=1}^k \dv_{i} a_i \not\equiv \sum_{i=1}^k \dv_{i} a'_i \pmod {2^{L(\dv_{r})+1}},$$ and therefore the corresponding
sums are distinct.
\end{proof}

\begin{proof}[Proof of Claim~\ref{cl: monotone}]
Notice that $f_n(\al) = f_1(\al)^n$, and $f_1(\al)>0$ for all $\al\in(0,1)$. For Part~\eqref{item: cl1}
it is therefore enough to show that
\[
f_1(\al): = \left(\frac{\al}{1-\al}\right)^{(1-\al)}2^{(1/2-\al)}
\]
is strictly monotone increasing. Taking logarithm it is enough to show
that
\[
g(\al):=\log f_1(\al) = (1-\al)(\log \al - \log (1-\al))+(\frac 1 2 -\al)\log 2,
\]
is strictly monotone increasing. Differentiate $g$ to get
\[\
g'(\al) = -\log \al +\log (1-\al) + \frac 1 \al - \log 2.
\]
For $\al\le \frac 1 2$, $ \log (1-\al)>\log \al$ and $\frac 1 \al > \log 2$,
and thus $g'(\al)>0$. By continuity of $g'$ at $\al=\frac 1 2$, there exists
$c_0>\frac 1 2$ such that $f'(\al)>0$ also for $\al\in [\frac 1 2, c_0)$, as
required.

For part~\eqref{item: cl2}, notice that $f_1(\tfrac12)=1$. Let $c>0$ be given. By continuity of $f_1$, there exists $\delta\in(0,c_0-1/ 2)$ such that
\begin{equation*}
f_1\left(\tfrac 1 2+\delta\right)<2^c.
\end{equation*}
Thus, for all $n\in\N$ we have $f_n(\frac 1 2+\delta)= f_1(\frac 1
2+\delta)^n < 2^{cn}$, as required.
\end{proof}

\subsection{Proof of Lemma~\ref{lem: crucial bound imp}}\label{sec: lemma}
In this section we prove Lemma~\ref{lem: crucial bound imp}. In the proof we keep using the notation introduced in the previous section.
We assume $\al n \in\N$. Let $\Z_n := \{ 0, 1, 2, \ldots, n -1\}$.

\subsubsection{Proof of item~\eqref{item: lem1}}
In order to bound $\Pro\Big(W= \al n \Big)$, we use
\begin{equation}\label{eq: lem1a-start}
\begin{aligned}\Pro\Big(W= \al  n \Big)&\le
\Pro\Big( \big | \{ i+w_i \pmod n\ :\ 1 \le i \le n \} \big | \le \al n \Big)
\end{aligned}
\end{equation}

For a set $A\subset \Z_n$, we write
\begin{equation}\label{eq: def of U}
U(A):=\Pro\Big(\{i +w_i \pmod n\ :\ 1 \le i \le n\} \subset A\Big).
\end{equation}

We then intend to show the following.
\begin{equation}\label{eq: lem1a-main}
\text{For every }A\subset \Z_n
\text{ of size }|A|=\al n\text{, we have }U(A)\le\al ^n.
\end{equation}

Part~\eqref{item: lem1} of Lemma~\ref{lem: crucial bound imp} would follow from \eqref{eq: lem1a-main} since
\begin{align}\label{eq: sum on A}
\Pro\Big(W = \al n\Big)  \le  \sum_{A: |A|=\al n} U(A) \le \binom{n}{\al n}\al^n\le \left(\frac{\al}{1-\al}\right)^{n(1-\al)},
\end{align}
where the leftmost inequality uses \eqref{eq: lem1a-start}, the middle one uses \eqref{eq: lem1a-main} and a union bound, and the rightmost one follows from
the well-known inequality of the binomial coefficient  $\binom{n}{k} \le \tfrac{n^n}{k^k (n-k)^{n-k}}$.

Towards showing \eqref{eq: lem1a-main}, let $A\subset \Z_n$ of size $|A|=\al n$.
Observe that, by the fact that $w_i$'s  are i.i.d., we have
\[
\Pro\big(\{i+w_i\pmod n\ :\ 1 \le i \le n\}\subset A\big) =\prod_{i=1}^n \Pro\big( i+w_i \pmod n \in A\big).
\]
We further observe that for every $a\in A$, we have
\[
\sum_{i=1}^n \Pro\Big(i+w_i \equiv a\pmod n\Big) = \sum_{i=1}^n \Pro\Big(w_i \equiv a - i\pmod n\Big) =\sum_{i=1}^n \Pro\Big(w_1\equiv a - i \pmod n \Big) = 1.
\]
Writing
\begin{equation}\label{eq: u_n}
u_i = u_i(A) := \Pro\big( i+w_i \pmod n \in A\big),
\end{equation}
we get that
\begin{equation}\label{eq: balance}
\sum_{i=1}^n u_i= \sum_{a\in A}\left(\sum_{i=1}^n \Pro\Big(i +w_i  \equiv a\pmod n \Big)\right) = |A|=\al n.
\end{equation}
We now solve the following maximization problem:
\begin{equation}\label{eq: max problem}
\text{maximize } U(A):=\prod_{i=1}^n u_i, \text{   under the constraints  } u_i \in [0,1], \;\sum_{i=1}^n u_i = \al n.
\end{equation}

By applying Jensen's inequality to the $\log$ function, we get
\[
\log U(A) =  \sum_{i=1}^n \log u_i \le n \cdot \log \left(\frac {\sum_{i=1}^n u_i}{n}\right) \le n\log \al,
\]
and so
\begin{equation}\label{eq: U less then alpha n}
U(A)\le \alpha^n,
\end{equation}
as required.\qed

\subsubsection{Proof of item~\eqref{item: lem2}}
To show part~\eqref{item: lem2} of Lemma~\ref{lem: crucial bound imp} it would suffice to show that
\begin{equation}\label{eq: item2main}
\text{$\exists\delta,\ep>0$, $\exists n_0\in\N$ s.t. every $\alpha\in(1/2-\delta,1/2+\delta)$, $n>n_0$ satisfy } \Pro\big(W=\alpha n)<e^{-\ep n}.
\end{equation}
To do so we shall use concentration arguments. We begin by showing that
\begin{equation}\label{eq: delta0def}
\E [W]\ge (\tfrac12 +\eta)n
\end{equation}
for some $\eta>0$.
To this end write
$$W=\sum_{i=1}^n\ind\{\forall j< i:w_i+i\neq w_j+j\},$$
and observe that
$$\Pro\big(\forall j<i:w_i+i\neq w_j+j\big)=\sum_{k\in\Z_+}\Pro(w_i=k)\prod_{j=1}^{i-1}\Pro(w_j\neq k+i-j)\ge\sum_{k\in\Z_+}\Pro(w_1=k)\prod_{j=1}^{\infty}\Pro(w_1\neq k+j).$$
Letting $p_i:=\Pro(w_1=i)$, we have
\begin{align}\label{eq: delta0bound}
\E[W]&\ge n \sum_{k\in\Z_+}\Pro(w_1=k)\prod_{j=1}^\infty\Pro(w_1\neq k+j)\ge  n \sum_{k\in\Z_+}p_k\Big(1 - \sum_{j\in\N}p_{k+j}\Big)\notag\\
     &= n \Big(1-\sum_{k<\ell \in \Z_+}p_k p_\ell\Big)=\frac{n}2\Big(1+\sum_{k\in\Z_+}p_k^2\Big).
\end{align}
Thus \eqref{eq: delta0def} is satisfied with $\eta:=\frac12\sum_{k\in\N}p_k^2$.

Next, we show that $W$ is concentrated around its expectation. To this end we
use the concentration properties of self-bounding functions of independent variables.
We write $f(w_1,\dots,w_n):= W$, $g_i(w_1,\dots,w_{i-1},w_{i+1},\dots,w_n):=|\{w_j +j \ :\ 1 \le j \le n, \  j \ne i\}|$, and observe that for all $i\le n$ we have,
$$ f(w_1,\dots,w_n)-g_i(w_1,\dots,w_{i-1},w_{i+1},\dots,w_n) \le 1.$$
$$ \sum_{i=1}^n \Big(f(w_1,\dots,w_n)-g_i(w_1,\dots,w_{i-1},w_{i+1},\dots,w_n)\Big)\le f(w_1,\dots,w_n).$$
We then apply \cite[Theorem 1 \& (7)]{LM}, to obtain that for every $\beta>0$
$$ \Pro\Big( W \le \E[W] - n\beta\Big)\le e^{-\beta^2\frac{n^2}{2\E[W]}}\le e^{-n \frac {\beta^2}2}.$$
Setting $\delta=\frac{\eta}2$ and $\ep=\frac{\eta^2}8$ and using \eqref{eq: delta0def} we observe that
for every $\alpha<\frac12 + \delta\le\frac{\E[W]}n-\frac{\eta}2,$ we have
$$\Pro\big(W=\alpha n \big)<\Pro\Big( W\le \E[W]  -  \tfrac{\eta n}{2} \Big)\le e^{ - n  \frac{\eta^2}{8}} \le 2^{-\ep n}.$$
This proves \eqref{eq: item2main} and hence completes the proof of part~\eqref{item: lem2} of Lemma~\ref{lem: crucial bound imp}.\qed

\section{High algebraic degree}\label{sec:high_alg_degree}

This section is dedicated to the proof of the following proposition, of which Lemma~\ref{lem:high_degree} is a straightforward consequence.

\begin{proposition}\label{prop:high_degree}
For any constant $B>0$, there exist constants $c, C, C'>0$, depending on $M$, such that for any $1\le d\le n$,
  \[\P(P\text{ has a double root $\alpha$ with $\deg(\alpha)\ge
    d$})\le Cn^{C'}\exp(-cd) + Cn^{- B}.\]
\end{proposition}

The proof of the proposition relies on the following consequence of Theorem~\ref{thm: main}.

\begin{lemma}\label{lem:integer_divisibility}
Let $P$ be the random polynomial as in \eqref{eq:poly_def}. Then there exist constants $C, \eps>0$
  such that for any positive integer $k$ and for $a \in \{ -2, 2\}$ we have
  \begin{equation*}
    \P\left(\text{P(a) is divisible by $k^2$}\right)\le C
    k^{-(1+\eps)}.
  \end{equation*}
\end{lemma}

\begin{proof}
Fix $a \in \{ -2,2\}$.  Let $k\ge 1$ be an integer and let $r$ be the integer satisfying
  $M2^r\le k^2<M2^{r+1}$. By conditioning on $\xi_r, \xi_{r+1}, \ldots, \xi_n$ we have
\begin{align}\label{eq:P3ad}
\P\left(P(a) \bmod k^2 = 0\right) \le \max_{m\in \mathbb Z} \P\left(
\sum_{j=0}^{r-1} \xi_j a^j \bmod k^2 = m\right) = \max_{m \in
\mathbb Z} \P\left(\sum_{j=0}^{r-1} \xi_j a^j  = m\right),
\end{align}
where the last equality follows from the fact that
$\left|\sum_{j=0}^{r-1}\xi_j a^j\right|\le M(2^r-1)$
deterministically and $k^2\ge M2^r$ by the definition of $r$.  From Theorem~\ref{thm: main}, it follows that there exists a constant $\eps \in (0, 1)$ such that
\begin{equation}\label{eq:P3bd}
 \max_{m \in
\mathbb Z} \P\left(\sum_{j=0}^{r-1} \xi_j a^j  = m\right) \le
\left(\frac{1}{\sqrt{2}} \right)^{r(1+\eps)}.
 \end{equation}
Combining \eqref{eq:P3ad} and \eqref{eq:P3bd} with the fact that
$r>2\log_2 k  - \log_2 M - 1$,  we conclude that
\[  \P\left(P(a) \bmod k^2 = 0\right)  \le 2^{1+ \log_2 M}  \left(\frac{1}{\sqrt{2}} \right)^{2(1+\eps) \log_2 k } = 2^{1+ \log_2 M}  k^{-(1+ \eps)}.\]
\end{proof}

We shall also use the following bound on the probability that $P$ has a root in close proximity to $\pm2$, whose proof we postpone to
Section~\ref{subs:roots near 2}.

Denote by  $B(z_0, r)$ the closed ball in $\C$ with center at $z_0$ and radius $r$.
\begin{lemma}\label{l:no_root_at_2}
For any constant $B>0$, there exists $K>0$ such that
\[ \P \Big( P \text{ has a zero  in }   B(2 , n^{ -K}) \cup B( -2 , n^{ -K})  \Big)   = O(n^{-B}). \]
\end{lemma}

Finally, we need a preliminary claim, bounding the number of roots far away from the unit circle.

\begin{claim}\label{clm:jensen}
Let $M \in \N$. For any $n \ge 1$ and any
non-zero polynomial $f$ in $\Z[x]$ of the form $f(z) = \sum_{i=0}^n a_i z^i$
with $ |a_i| \le M$ for all $0 \le i \le n$,  the number of
zeros of $f$ with modulus at least $\tfrac{3}{2}$  is at most
$64M$.
\end{claim}
\begin{proof}
Assume, without loss of generality, that  $|a_n| \ne 0$. Let $\tilde
f(z) = z^n f(z^{-1}) = \sum_{i=0}^n a_i z^{n-i}$ be the reciprocal
polynomial of $f$. Denote by $N(f)$ the number of $z\in\C$ for which
$f(z) = 0$ and $|z|\ge\frac{3}{2}$. Then $N(f)$ is also the number
of $z\in\C$ for which $\tilde f(z)=0$ and $|z|\le \frac{2}{3}$.
Noting that $|\tilde{f}(0)|=|a_n| \ge 1$ we may apply Jensen's formula (see,
e.g., \cite[Chapter 5.3.1]{A78}) and obtain for any $r> \frac{2}{3}$
that
\begin{equation*}
  \max_{0\le \theta\le 2\pi} \log|\tilde{f}(re^{i\theta})| \ge
  \frac{1}{2\pi}\int_0^{2\pi} \log|\tilde{f}(re^{i\theta})|d\theta =
  \log|\tilde{f}(0)| + \sum_{z\colon \tilde{f}(z)=0,\, |z|\le r} \log\left(\frac{r}{|z|}\right) \ge N(f)
  \log\left(\frac{r}{2/3}\right).
\end{equation*}
Observe that when $r<1$ we have $|\tilde{f}(re^{i\theta})|\le
\frac{M}{1-r}$ for all $\theta$. Thus
\begin{equation*}
  N(f)\le \frac{M}{(1-r)\log(3r/2)},\quad \frac{2}{3}<r<1
\end{equation*}
and substituting $r=0.8$, say, we obtain that $N(f)\le 64M$, as required.
\end{proof}

{\bf Proof of Proposition~\ref{prop:high_degree}.}
 Fix $1 \le d \le n$. Let
  $\alpha$ be an algebraic number of degree $\deg(\alpha) = d$ and let $f_\ga$ be the associated minimal poly of $\alpha$ in $\Z[x]$.
Suppose that $\alpha$ is a double root of $P$.  Note  that  $\alpha$ cannot be a
multiple root of $f_\ga$, since, otherwise, $\alpha$ is also a root of
the polynomial $f_\ga'$ whose degree is strictly smaller than  $d$, violating the definition of $\deg(\alpha)$. This implies that $f_\ga^2$ divides $P$
in $\mathbb{Z}[x]$ (by Gauss's lemma). In particular,
\begin{equation} \label{eq:double_root_divisibility_implication}
\text{ the integer } P(a) \text{ is
divisible by } f_\ga(a)^2, \quad \text{ for } a = \pm 2.
\end{equation}
Next we obtain a suitable lower bound for $\max\{ |f_\ga(2)|, |f_\ga(-2)|\}$.
 Denote by $C(\alpha)$ the set of
  algebraic conjugates of $\alpha$ (i.e., the set of roots of
  $f_\ga$). Each of these conjugates of $\alpha$ must also be a root of $P$.
 So, by
Claim~\ref{clm:jensen}, all but at most $64M$ of the $\beta\in
C(\alpha)$ satisfy
$|\beta|\ge \frac{3}{2}$. Therefore, we have
\begin{align*}  |f_\ga(-2)|\cdot| f_\ga(2)| &= \prod_{\beta \in C(\ga)} |\beta +2| \cdot|\beta - 2|\\
 &\ge \left(\prod_{\beta \in C(\ga), |\beta| \le 3/2} |\beta^2- 4|  \right)\left(  \min_{\beta \in C(\ga)}  |\beta  + 2|  \wedge 1  \right)^{64M}   \left(  \min_{\beta \in C(\ga)}  |\beta  - 2|    \wedge 1 \right)^{64M}
\end{align*}
Let $B>0$ be given as in Proposition~\ref{prop:high_degree} and let $K = K(B)>0$ be as given by Lemma~\ref{l:no_root_at_2}.  Let $\mathcal{E}$ be the event that
 there is at least one root of $P$ within a distance of $n^{ - K}$ from either $-2$ or $2$. Note that the event $\mathcal{E}$ does not depend on $\alpha$.  On the event $\mathcal{E}^c$,
\[ \min_{\beta \in C(\ga)}  |\beta  -a|  \ge \min_{z: P(z) =0}  |z   - a|  \ge n^{-K} \quad \text{for any } a \in \{ -2, 2\}.\]
On the other hand, $|\beta^2- 4| \ge \tfrac74$ for any $|\beta| \le \tfrac32$. Putting these ingredients together,
we conclude that on the event $\mathcal{E}^c$,
\begin{equation*}
 |f_\ga(-2)|\cdot| f_\ga(2)| \ge \left( \tfrac74 \right)^{d- 64M}  n^{-128 KM}.
  \end{equation*}
 Consequently, we obtain the following lower bound
 \begin{equation}\label{eq:lower_bound_of_f_at_2}
 \max \big \{ |f_\ga(2)|, |f_\ga(-2)| \big\}    \ge  c_1 \exp(c_2 d) n^{ - C_1},  \ \ \text{ on } \mathcal{E}^c,
 \end{equation}
 where $c_1 := (\tfrac{7}{4})^{-32M}>0, c_2 :=\tfrac12\log(\tfrac{7}{4})$ and $C_1: = 64KM$.  From \eqref{eq:double_root_divisibility_implication} and \eqref{eq:lower_bound_of_f_at_2},
we arrive at the inclusion of events
\begin{equation*}
 \left\{\text{$\alpha$ is a double root of $P$}\right\}\subseteq  \mathcal{E} \bigcup_{a \in \{-2,2\}} \left\{\text{$P(a)$
  is divisible by $k^2$ for some integer $k\ge
  c_1 e^{c_2 d} n^{ - C_1}$}\right\}.
\end{equation*}
By Lemma~\ref{l:no_root_at_2}, $\P(\mathcal{E})  = O(n^{-B})$.  On other hand, by Lemma~\ref{lem:integer_divisibility}, we deduce that
\[ \left\{\text{$P(a)$
  is divisible by $k^2$ for some integer $k\ge
  c_1 e^{c_2 d} n^{ - C_1}$}\right\} \le C_2 ( e^{c_2 d } n^{ - C_1})^{-\eps} = C_2 e^{ - c_3 d} n^{C_3}, \]
for suitable constants $c_3, C_2, C_3>0$.
Proposition~\ref{prop:high_degree} follows. \qed

\subsection{Roots near $\pm2$}\label{subs:roots near 2}

In this section we prove Lemma~\ref{l:no_root_at_2}. We shall require the following.

\begin{lemma} \label{l:value_at_2}
For any constant $B>0$, there exists $C>0$ such that for $a \in \{ -2, 2\}$,
\[ \P( |P(a)| \le n^{ - C}2^{n}   )   = O(n^{-B}). \]
\end{lemma}
\begin{proof}
We prove the lemma for the case $a=2$, as the argument for the case $a = -2$  is nearly identical.  Set $C_1 = \lceil\log_3M\rceil$.  Define a subset of indices $J$ as
\[ J = \{ j \in \{ 0, 1, \ldots, n\}: j \ge n - C_1 \log_2 n, \text{ and } j \text{ is divisible by } \lceil\log_2(2M+1)\rceil\}.\]
By conditioning on the random variables $\xi_j, j \not \in J$, we deduce that
\begin{align*}
\P( |P(2)| \le n^{ - C}2^{n}   )  \le \sup_{z \in \R} \P \left( \big |\sum_{ j \in J} \xi_j 2^j - z \big| \le n^{ - C}2^{n}  \right).
\end{align*}
Note that for any two different values of the random vector $(\xi_j)_{j \in J}$ in $\{ 0, \pm 1, \ldots, \pm M\}^{|J|}$, the values of the sum $\sum_{ j \in J} \xi_j 2^j$ differ by at least
$\tfrac12 2^{ n - C_1 \log_2 n} = \tfrac12 n^{ -C_1}2^n $. Thus if we choose $C = C_1+1$, then  for any fixed $z \in \R$, there exists at most one value of the random vector $(\xi_j)_{j \in J}$ in $\{ 0, \pm 1, \ldots, \pm M\}^{|J|}$ such that $|\sum_{ j \in J} \xi_j 2^j - z| \le n^{ - C}2^{n}$. Now by Assumption~\ref{eq:coefficient_condition}, we conclude that
\[ \P \left( \big |\sum_{ j \in J} \xi_j 2^j - z \big| \le n^{ - C}2^{n}  \right) \le 2^{- |J|} \le 2^{ - \lfloor (4M)^{-1} C_1 \log_2 n -1 \rfloor} = O(n^{-B}). \]
\end{proof}

\begin{proof}[Proof of Lemma~\ref{l:no_root_at_2}]
By Lemma \ref{l:value_at_2}, there exists a constant $C>0$ such that
\begin{equation}\label{eq:bd_at_2}
 \P( |P(\pm 2)| \ge n^{ - C}2^{n}   )   = 1 - O(n^{-B}).
 \end{equation}
By mean value theorem and the triangle inequality, for any $z \in \C$ such that $|z| \le n^{-1}$,
\begin{equation}\label{eq:mean_value}
 |P(2+ z)| \ge |P(2)| - \sup_{w \in B(2, n^{-1})} |P'(w)| \cdot |z|.
 \end{equation}
We can now bound the derivative of the polynomial $P'$ in $B(2, n^{-1})$ by
\begin{equation}\label{eq:derv_bdd}
\sup_{w \in B(2, n^{-1})} |P'(w)|  \le \sum_{i=0}^n M i(2+ n^{-1})^{i-1} \le 3Mn 2^n.
\end{equation}
Plugging in the bound \eqref{eq:bd_at_2} and \eqref{eq:derv_bdd} in \eqref{eq:mean_value}, we deduce that, for any $|z| \le n^{ - (C+2)}$ and for sufficiently large $n$,
\[  |P(2+ z)| \ge n^{ - C}2^{n}   - 3Mn 2^n  \cdot n^{ - (C+2)}  >0,\]
with probability  $1 - O(n^{-B})$. The lemma is then obtained by taking $K = C+2$.
\end{proof}

\section{Roots far from the unit circle}

In this section we prove Lemma~\ref{lem: med deg large house}. We begin by obtaining the following bound on the probability that $P$ has a particular root $\alpha$
far from the unit circle.
\begin{lemma}\label{lem:off_circle}
For each $\alpha\in \ALG$, we have
  \begin{equation*}
    \Pro\Big(P(\alpha)=0\Big)\le \exp \Big( {-\frac{n\log 2}{\lceil \log (M+1) /|\log |\alpha||\rceil}} \Big).
  \end{equation*}
\end{lemma}
\begin{proof}
Assume that $|\alpha|>1$. Let $\ell$ be the minimal positive integer for which
\begin{equation}\label{eq:j_0_prop}
  |\alpha|^{\ell} > M+1.
\end{equation}
Write $P(z) = P_1(z) + P_2(z)$ with
\begin{equation*}
  P_1(z):=\sum_{k=0}^{\lfloor n/\ell\rfloor} \xi_{k\ell} z^{k
  \ell}\quad\text{and}\quad P_2(z):=P(z) - P_1(z).
\end{equation*}
Since $|\xi_i| \le M $ for all $i$, the map
\begin{equation*}
  \big( \xi_0,\xi_{\ell},\xi_{2 \ell}, \dots,\xi_{\lfloor n/\ell\rfloor \ell} \big) \mapsto P_1(\alpha) : \{-M,\dots,M\}^{\lfloor n/\ell\rfloor+1}\to\C
\end{equation*}
is one-to-one. Thus, as
$P_1(\alpha)$ and $P_2(\alpha)$ are independent, we have
\begin{align*}
  \Pro\Big(P(\alpha)=0\Big) &= \E\left[\Pro\Big(P(\alpha)=0\Big)\ |\ P_2(\alpha))\right] =\\
  &= \E\left[\Pro\Big(P_1(\alpha) =
  -P_2(\alpha)\ |\ P_2(\alpha)\Big)\right]\le
  \left(\max_{x\in\{-M,\dots,M\}}\Pro\Big(\xi_0=x\Big)\right)^{\lfloor
  n/\ell\rfloor+1}.
\end{align*}
Assumption~\eqref{eq:coefficient_condition} and the
definition of $\ell$ imply that
\begin{equation*}
  \Pro\Big(P(\alpha)=0\Big)< \left(\tfrac{1}{2}\right)^{\lfloor
  n/\ell\rfloor+1}\le e^{-\frac{n\log 2}{\ell}} = e^{-\frac{n\log 2}{\lceil \log (M+1) / \log
  |\alpha|\rceil}}.
\end{equation*}
The case when $|\alpha|< 1$ can be handled similarly. This proves the lemma.
\end{proof}
We also make the following simple observation.
\begin{observation}\label{obs:root_house_bdd}
Let $\alpha$ be a root of $P$. Then
\begin{enumerate}
\item The leading coefficient of the \emph{associated minimal polynomial} of $\alpha$ in $\Z[x]$ is at most $M$,
\item $\Lambda(\alpha) \le M+1.$
\end{enumerate}
\end{observation}
\begin{proof}
Write
$f_{\alpha}$ for the associated minimal polynomial of $\alpha$ in $\Z[x]$ and denote the leading coefficient of $f_\alpha$ by $m_\alpha$. By Gauss's lemma (see, e.g.,  \cite[Proposition~3.4]{Ar}),   $f_\alpha | P$ in $\mathbb{Z}[x]$ and, in particular, $m_{\alpha}$, the leading coefficient of $P$, divides $\xi_n$. Since $|\xi_n|\le M,$ we get that $m_\alpha \le M$.
%

The fact that $\Lambda(\alpha)\le M+1$ is a direct consequence of Rouch\'e Theorem.
\end{proof}

\begin{proof}[Proof of Lemma~\ref{lem: med deg large house}]
Let us first estimate the number of algebraic numbers $\alpha$ such that $\alpha$ is a root of some random polynomial $P$ of degree $n$ and  $\deg(\alpha) \le C_0 \log n$. Write
$f_{\alpha}$ for the associated minimal polynomial of $\alpha$ in $\Z[x]$ and denote, as usual,  the leading coefficient of $f_\alpha$ by $m_\alpha$. If $\alpha_1 = \alpha, \alpha_2, \cdots, \alpha_{\deg(\alpha)}$ are conjugates of $\alpha$, we can express $f_\alpha$ as
$$f_{\alpha}(x)=m_\alpha(x-\alpha_1)\cdots(x-\alpha_{\deg(\alpha)})=\sum_{j=0}^{\deg(\alpha)} a_jx^{n-j}.$$
Therefore, by Observation~\ref{obs:root_house_bdd}, we have following crude bound on the coefficients of $f_\alpha$,
$$|a_j|= |m_\alpha\sum_{i_1<\dots<i_j}\alpha_{i_1}\cdots\alpha_{i_j}|\le \left |m_\alpha \binom{\deg(\alpha)}{j}\Lambda(\alpha)^j \right |\le e^{C' \log n}$$
for some $C'>0$ depending on $C_0$ and $M$. Since $a_j$ has to be an integer, there are at most $e^{C \log^2 n}$ possibilities for $f_\alpha(x)$ for some constant $C>0$.

Now, by Lemma~\ref{lem:off_circle}, if $\alpha \in \mathbb{A}$ with $\Lambda(\alpha)>1+\tfrac{C_1}{\log n}$, then
\begin{equation}\label{eq:far_each}
\max_{\alpha\in A}\Pro\Big(P(\alpha)=0\Big)
=e^{-\Omega(n/\log{n})},
\end{equation}
A simple union bound over such $\alpha$ (or, more precisely, over the minimal polynomials $f_\alpha$) yields the lemma.
\end{proof}

\section{Roots near the Unit circle}\label{sec:roots_near_unit}
In this section we prove Lemma~\ref{lem: med deg small house}. Recall that for $\alpha \in \mathbb{A}$,  $m_\alpha$ denotes the leading coefficient of the associated minimal polynomial of $\alpha$ in $\Z[x]$. Fix $C_0, C_1>0$. By Observation~\ref{obs:root_house_bdd},
we need consider for Lemma~\ref{lem: med deg small house}  the following set of potential roots of $P$.
 \[ A:= \Big\{ \alpha\in\mathbb{A}:\   4<\deg(\alpha)\le C_0\log n\text{ and }\Lambda(\alpha)\le1+\tfrac{C_1}{\log n}  \text{ and } m_\alpha \le M \Big\}.\]
To prove Lemma~\ref{lem: med deg small house}, we employ the following  union bound
\begin{equation}\label{eq:med_deg_union_bound}
\Pro\Big(\exists \alpha\in A : P(\alpha)=0\Big)\le |A|\cdot\max_{\alpha\in A}\Pro\Big(P(\alpha)=0\Big)
\end{equation}
and then proceed to provide upper bounds on $\max_{\alpha\in A}\Pro\big(P(\alpha)=0\big)$ and on the cardinality of the set $A$. This is done using
Lemma~\ref{lem:med_deg} and Lemma~\ref{lem:poly_count} below, whose proofs are presented in Sections~\ref{subs:med_deg_each_root} and~\ref{subs:med_deg_not_many_poly} respectively.
\begin{lemma}\label{lem:med_deg}
Let $\alpha$ be an algebraic number of degree at least $5$. Then for every $\epsilon>0$ there exists $C>0$ such that
\begin{equation}\label{eq:lem_med_deg}
\Pro(P(\alpha)=0)\le C n^{-\frac52 + \epsilon}.
\end{equation}
\end{lemma}
For any polynomial $f \in Z[x]$, let  $\Lambda(f)$ be the maximum modulii of the roots of $f$.
\begin{lemma}[counting integral polynomials with small houses]\label{lem:poly_count}
Let  $b > 0$ and let $a\in\N$. Then, for all $d$ sufficiently large, the number of polynomials  $f \in \Z[x]$ of degree $d$ with leading coefficient $a$ such that
\[ \Lambda(f)  < 1 + \frac{b \log d}{a  d}\]
is less than $\exp( (a d)^{2/3 + b} )$.
\end{lemma}

Note that for every algebraic number $\al \in A$, its associated minimal polynomial in $\Z[x]$ has degree at most $C_0 \log n$ and its leading coefficient  is bounded by $M$.
Applying Lemma~\ref{lem:poly_count} to each $a \in \{1, 2, \ldots, M\}$  and each degree $1 \le  d \le C_1 \log n$ with $b=\tfrac16$, we obtain that for every $\epsilon>0$,
 \begin{equation}\label{eq:med_deg_not_many_poly}
|A|=o(n^\epsilon).
\end{equation}

Plugging \eqref{eq:lem_med_deg} and \eqref{eq:med_deg_not_many_poly} into \eqref{eq:med_deg_union_bound}, the Lemma~\ref{lem: med deg small house} follows.\qed

\subsection{Each root of low degree is unlikely}\label{subs:med_deg_each_root}

In this section we prove Lemma~\ref{lem:med_deg}. The proof follows closely the proof of \cite[Lemma~1]{KZ} adapted for our case (that is, when the random variables  $\xi_i$'s  are
not Bernoulli random variables).  The main ingredient of the proof is the `inverse Littlewood-Offord type theorem'' of Tao and Vu \cite[Theorem 1.9]{TV}, whose specialization for our case is the following.
\begin{theorem}[Tao and Vu (2010)]\label{thm: Tau Vu}
Let $(\eta_i)_{0 \le i \le n}$ be i.i.d.\ $\Ber(\tfrac12)$ random variables. Let $A, \delta > 0$ and let $(z_i)_{0 \le i \le n}$ be complex numbers such that
$$\max_{ z \in \C} \Pro\Big(\sum_{i=0}^n \eta_iz_i = z\Big)\ge n^{-A}.$$
Then there exists  a symmetric generalized arithmetic progression (GAP), all of whose elements are distinct, of rank $r \le 2A$
which contains all but $O_{A, \delta}(n^{1-\delta})$ of  $z_i$'s (counting multiplicities).
\end{theorem}

Recall that in our context a \emph{symmetric GAP}  $Q$ of rank $r$ is a set of the form
\begin{equation}\label{eq:GAP}
Q=\left\{\sum_{i=1}^r n_iu_i :\ n_i\in[-N_i,N_i] \cap \Z, \ \forall i = 1,\dots, r \right\},
\end{equation}
where  the dimensions $N = (N_1, N_2, \ldots, N_r)$ are $r$-tuple
of positive integers and the steps $u = (u_1, u_2, \ldots, u_r)$ are $r$-tuple of elements in $\C$. In particular,  if $v_1 \dots,v_{r+1}$ are elements of a GAP of rank $r$,
then there exist nontrivial integer coefficients $(q_1,\dots,q_{r+1} )\in \Z^{r+1}, (q_1,\dots,q_{r+1} ) \ne \mathbf{0}$  such that
$q_1v_1+ q_2v_2 + \ldots + q_{r+1} v_{r+1}=0$.

\begin{proof}[Proof of Lemma~\ref{lem:med_deg}]

Let $\alpha$ be an algebraic number of degree $d\ge5$ and let $\epsilon>0$.
Assume towards obtaining a contradiction that
\begin{equation}\label{eq:med_deg_contra}
\Pro\Big(P(\alpha)=0\Big)> n^{-\frac52 + \epsilon}.
\end{equation}
 Proposition~\ref{propos:mixture-reduction} allows us to represent the random variable $\xi_j$ as $\xi_j  = I_j + \Delta_j \eta_j$ where $(I_j, \Delta_j)_{0 \le j \le n}$ are i.i.d.\ random vectors taking values in $\Z \times \N$ and $(\eta_j)_{0 \le j \le n}$ are i.i.d.\ $\mathrm{Ber}(\tfrac12)$,  independent of $(I_j, \Delta_j)_{0 \le j \le n}$. Conditioning on $(I_j, \Delta_j)_{0 \le j \le n}$ yields
\begin{equation}\label{eq:LOaveraging}
 \Pro( P(\alpha) = 0)  =  \E \Pro \big( \sum_{j=0}^n \Delta_j \alpha^j \eta_j = - \sum_{j=0}^n I_j \alpha^j|  (I_j, \Delta_j)_{0 \le j \le n} \big) \le \E \sup_{z \in \C} \Pro \big( \sum_{j=0}^n \Delta_j \alpha^j \eta_j = z |  (\Delta_j)_{0 \le j \le n} \big).
 \end{equation}
From \eqref{eq:med_deg_contra} and \eqref{eq:LOaveraging}, it follows that there exists a vector $(d_0, d_1, \ldots, d_n) \in \N^{n+1}$ such that
\[  \sup_{z \in \C} \Pro \big( \sum_{j=0}^n d_j \alpha^j \eta_j = z  \big) >  n^{-\frac52 + \epsilon}.\]
We now apply Theorem~\ref{thm: Tau Vu} with $A = \tfrac52 - \eps$ and $\delta=\tfrac12$ and $z_j = d_j \alpha^j $ to obtain a   symmetric GAP $Q$ of rank $B\le2A<5$ such that all but $O(\sqrt{n})$ many of the coefficients $d_j \alpha^j$ belong to $Q$. Therefore, for large enough $n$, there exists $j_0\in\{0,1, \dots, n\}$ for which $d_{{j_0}+k}\alpha^{{j_0}+k} \in Q$ for all $k=0,1,\ldots, 4$. Since the rank of $Q$ is at most $4$, there exists a nontrivial integer linear combination that annihilates the vector $(d_{{j_0}+k}\alpha^{{j_0}+k})_{0 \le k \le 4}$. Hence the algebraic degree of $\alpha$ is at most $4$, in contradiction with our assumption. Hence, the lemma follows.
\end{proof}

\subsection{There are not many low degree polynomials with small house}\label{subs:med_deg_not_many_poly}

This section is dedicated to the proof of Lemma~\ref{lem:poly_count}.

\begin{proof}[Proof of Lemma~\ref{lem:poly_count}]
The proof of the lemma is an adaptation of the proof of \cite[Theorem 1]{Du} of Dubickas to the case of non-monic polynomials.
Fix $b>0$, and write $F_{a,d}$ for the collection of polynomials $f \in \Z[x]$ of degree $d$ with leading coefficient $a$ which satisfy
\[ \Lambda(f)  < 1 + \frac{b \log d}{a  d}.\]

In the proof we also make use of the classical Newton identities,
known also as Newton–-Girard formulae (See \cite{Dk} for a modern proof).
Let $f(x)=\sum_{i=0}^d a_i x^{d-i}$ be a polynomial of degree $d$ in $\Z[x]$ with roots $\alpha_1(f), \alpha_2(f), \ldots, \alpha_d(f)$.
Define, for $k \ge 1$,
\[ S_k = S_k(f) = \sum_{j=1}^d \alpha_j(f)^k. \]
\begin{lemma}[Girard, 1629]\label{lem:Newton-Girard}
 for each $1 \le k \le d$,
\begin{equation}\label{eq:newton}
 a_0S_k + a_1S_{k -1} + . . . + a_{k - 1}S_1 + ka_k = 0.
 \end{equation}
\end{lemma}
We further observe that if $g(x)= \sum_{i=0}^{k-1} a_i x^{d-k}+\sum_{i=k}^d b_i x^{d-i}$ with $b_k\neq a_k$ then
for all $i<k$ we have $S_i(f) = S_i(g)$ while
\begin{equation}\label{eq:separation}
|S_k(f) - S_k(g)| = \frac{k|a_k - b_k|}{a_0}  \ge \frac{k}{a_0}.
\end{equation}
With a slight abuse of notation we write $S_k(w) = w_1^k + w_2^k + \cdots + w_d^k$, for $w = (w_1, \ldots, w_d)  \in \C^d$.

Let $b>0$ be given. We say that a set $W\subset \C^d$ is $(a, d)$ {\em admissible} if the following two conditions are satisfied:
\begin{itemize}
\item{(Boundedness)} We have $\max_{1\le i\le d}|w_i| <  1 + \tfrac{b \log d}{ a d}$ for all $w \in W$.
\item{(Separation)}  For any two distinct $u, v \in W$ we have \[ \max_{1 \le k \le d} \frac{a} {k}   \Big | \Re( S_k (u)) - \Re(S_k(v)) \Big|  \ge 1.\]
\end{itemize}
Let $$S=\{(\alpha_1(f),\dots ,\alpha_d(f))\ :\ f\in F_{a,d}\}.$$
From \eqref{eq:separation} and from the definition of $F_{a,d}$, we deduce that the set $S$
is $(a, d)$ admissible. To conclude the proof we bound the maximal size of any $(a, d)$ admissible set.
In \cite[Theorem 2]{Du}, Dubickas obtained such a bound for the special case when $a=1$
using an elementary but clever application of
volume formulas of polytopes, and classical estimates on the number of Gauss integers in a circle.
 \begin{theorem}[Dubickas, 1999]\label{thm:dub}
The size of any $(1, \ell)$ admissible set is $O_b( \exp(\ell^{2/3+ b}))$.
\end{theorem}
We now use Dubickas' result as a blackbox to bound the cardinality of $(a, d)$ admissible set.
Let $W$ be an $(a, d)$ admissible set, and write $\widehat W $ for the image of $W$ in $\C^{a d}$ under the repetition map
\[ w \in \C^d   \mapsto \hat w := ( \underbrace{w, w, \ldots, w}_{a \text{  times}} ) \in \C^{a d}. \]
Clearly, $||\hat w||_\infty = || w||_\infty  <  1 + \tfrac{b \log d}{ a d} \le 1 + \tfrac{b \log (a d)}{ a d}  $ for all $w \in W$.  Furthermore, for every distinct $\hat u, \hat v \in \widehat W$ we have
\begin{align*}
  \max_{ 1 \le k \le a d} \frac{1} {k}   \big | \Re( S_k (\hat u)) - \Re(S_k( \hat v)) \big| &\ge  \max_{ 1 \le k \le d} \frac{1} {k}   \big | \Re( S_k (\hat u)) - \Re(S_k( \hat v)) \big|\\
&= \max_{ 1 \le k \le d} \frac{a} {k}   \big | \Re( S_k ( u)) - \Re(S_k( v))| \ge 1,
\end{align*}
where the last inequality follows from the fact $u\ne v \in W$ and $W$ is an  $(a, d)$ admissible set.  Hence $\widehat W$ is $(1, a  d)$ admissible. Therefore applying Theorem~\ref{thm:dub} with $\ell=a d$ implies
$$|F_{a,d}| = |S| =O_b( \exp((ad)^{2/3+ b})),$$ as required.
\end{proof}

\section{Unimodular roots with bounded degree}\label{sec:tiny_deg}
In this section we will prove Lemma~\ref{lem: small deg small house}. Note that there are only finitely many irreducible polynomials in $\Z[x]$ of degree at most $4$ with leading coefficients bounded by $M$ in absolute value whose roots are all within distance $M+1$ from the origin. In particular, for large enough $n$, all such polynomials whose roots are in distance $1+\frac{C_1}M$ from the origin, have, in fact, all roots on the unit circle. Thus to prove the lemma it would suffice to show that for any fixed $\alpha \in \mathbb{A}$ on the unit circle such that $\deg(\alpha) \le 4$ and $\alpha \ne  \pm 1$,
\begin{equation}\label{eq:low_deg_double_root_on_circle}
 \P(\ga \text{ is a double root  of } P) = o(n^{-2}),
 \end{equation}
Lemma~\ref{lem: small deg small house} will then follow by applying a simple union bound. If $\alpha$ is an algebraic {\em integer}, then it has to be a root of unity if $\alpha$ and all of its conjugates lie on the unit circle. However, in general, there are examples of  algebraic numbers such that all of their conjugates are on the unit circle yet they are not roots of unity. For example, consider the quadratic polynomial $3x^2 - x +3$ whose roots are $\tfrac{1\pm \sqrt{-35}}{6}$. In fact, any polynomial $\sum_{i=0}^m b_i x^i$ in $\Z[x]$ that is self-reciprocal (i.e., $b_i = b_{m-i} \ \forall \ i$) which satisfies the condition $|b_m| > \tfrac12 \sum_{k=1}^{m-1} |b_k|$ has all its root on the unit circle \cite{LL04}. Thus, first begin by addressing roots which lie on the unit circle but which are not root of unity.
  \begin{lemma}[unimodular roots that are not roots of unity]\label{lem:unimodular_but_not_root_unity}
Let $\alpha \in \mathbb{A}$ be such that $| \alpha |=1$ but $\alpha$ is not a root of unity (i.e., $\alpha^m  \ne 1$ for all $m \in \N$). Then under Assumption~\ref{eq:coefficient_condition},
 \[ \P \big ( P(\alpha) = 0 \big)  = O(n^{-5/2}).  \]
 \end{lemma}
 The proof of Lemma~\ref{lem:unimodular_but_not_root_unity} is a straightforward application of the following well-known result due to Hal\'asz (see \cite[Corollary 7.16]{TVbook}, \cite[Corllary 6.3 and Remark 3.5]{NV13}).
 \begin{lemma}[Hal\'asz] \label{lem:halasz}
 Let $G$ be an infinite Abelian group.  Let $m \ge 1$ and $a_1, a_2, \ldots, a_m \in G$ and let $\eps_1, \eps_2, \ldots, \eps_m$ be i.i.d.\  with $\P(\eps_j =1) = \P(\eps_j =0) = 1/2$. Fix $\ell \in \N$ and let $R_\ell$ be the number of solutions of the equation $a_{i_1} + a_{i_2} + \cdots + a_{i_\ell} = a_{j_1}+ a_{j_2} + \cdots + a_{j_\ell}$. Then
 \[ \sup_x \P \big( \sum_{i=1}^m a_i \eps_i = x \big ) = O(n^{-2 \ell -\tfrac12} R_\ell ).\]
  \end{lemma}
 \begin{proof}[Proof of Lemma~\ref{lem:unimodular_but_not_root_unity}]
Applying Proposition~\ref{propos:mixture-reduction} can represent the random variables $(\xi_j)_{0 \le j \le n}$ as $\xi_j  = I_j + \Delta_j \eps_j$ where $(I_j, \Delta_j)_{0 \le j \le n}$ are i.i.d.\ random vectors taking values in $\Z \times \N$ and $(\eps_j)_{0 \le j \le n}$'s are i.i.d.\ $\mathrm{Ber}(\tfrac12)$,  independent of $(I_j, \Delta_j)_{0 \le j \le n}$. Now by conditioning on $I_j$ and $\Delta_j$'s,  we have
\begin{align*}
\P \Big( \sum_{j=0}^n \xi_j \alpha^j = 0\Big)  \le \max_{x } \E \P \Big( \sum_{j=0}^n (I_j + \Delta_j \eps_j) \alpha^j = x \Big| (I_j, \Delta_j)_{0 \le j \le n}     \Big) \\
\le  \E  \max_{x }   \P \Big( \sum_{j=0}^n  \Delta_j \eps_j \alpha^j = x \Big| (\Delta_j)_{0 \le j \le n}     \Big)
\end{align*}
Where expectations are taken on the vector $(I_j, \Delta_j)_{0 \le j \le n}$.
Fix  an integer $q$ in the support of the random variable $\Delta_j$ and let $\eta:= \P(\Delta_j = q) >0$. Let $T$ denote the random set of indices defined by $T = \{ 0 \le j \le n: \Delta_j  = q\}$. Again, conditioning on $(\eps_j)_{ j \not \in T}$, write
\begin{align*}
\E  \max_{x }   \P \Big( \sum_{j=0}^n  \Delta_j \eps_j \alpha^j = x \Big| (\Delta_j)_{0 \le j \le n}     \Big) &\le \E  \max_{x }   \P \Big( \sum_{j \in T}  \Delta_j \eps_j \alpha^j = x \Big| T    \Big) \\
&\le  \E  \max_{x }   \P \Big( \sum_{j \in T}  \eps_j \alpha^j = x \Big| T    \Big).
\end{align*}
We are left with showing that for any deterministic set of indices $T \subseteq \{0,1, \ldots, n\}$,  we have $$\max_{x }  \P \Big( \sum_{j \in T}  \eps_j \alpha^j = x  \Big) = O(|T|^{-5/2}).$$
 Applying Hal\'asz's result (Lemma~\ref{lem:halasz}) with coefficients $(\alpha_j)_{j \in T}$ in $\C$ and $\ell=2$, we count the number of solutions of the equation
\begin{equation}\label{eq:h}
\alpha^{i} + \alpha^{j} = \alpha^{k} + \alpha^{l},
\end{equation}
 where $i, j, k, l$ are arbitrary indices in $T$. Taking the absolute value on the both sides of \eqref{eq:h} and using the fact that $|\alpha|=1$, we have
 $|1 + \alpha^{j-i}|  = |1 + \alpha^{l-k}|$, which implies that either $\alpha^{j-i} = \alpha^{l-k}$ or $\alpha^{j-i} = \alpha^{k-l}$, or equivalently $j-i = \pm(l-k)$. In case that $j-i = l-k$, we may write \eqref{eq:h} as  $\alpha^i (1 + \alpha^{j-i})  =  \alpha^k (1 + \alpha^{j-i})$, or equivalently as $(\alpha^i  - \alpha^k)(1 + \alpha^{j-i}) = 0$. Since $\alpha$ is not a root of unity, then $1+ \alpha^{j-1} \ne 0$. So, we deduce that $\alpha^i = \alpha^k$ which implies that $i = k$. Plugging it in back in $j-i = l-k$, we also have $j=l$. Similarly, for the case $j-i = k -l$, we end up with the equation $(\alpha^i  - \alpha^l)(1 + \alpha^{j-i}) = 0$, which, in turn, implies that $i = l$ and $j = k$. Hence, we conclude that $R_2 \le 2 |T|^2$ and the claim follows.

From the claim, we obtain that
\[ \E  \max_{x }   \P \Big( \sum_{j \in T}  \eps_j \alpha^j = x \Big| T    \Big)  \le \E \min \Big( 1, \frac{C}{|T|^{5/2}}\Big) = O(n^{-5/2}) + \P \Big( |T| \le \tfrac{\eta}{2} n \Big). \]
Note that $|T|$ has the distribution of a Binomial random variable with $n+1$ trials and success probability $\eta$.
From the standard result on the concentration of Binomial random variable, we know that there exists a constant $c>0$, depending on $\eta$, such that
$ \P \Big( |T| \le \tfrac{\eta}{2} n \Big) \le e^{-c(n+1)}$. This completes the proof of the lemma.
 \end{proof}
Next, we consider roots of unity.
  \begin{lemma}[roots of unity]
 \label{lem:roots_of_unity}
 Suppose Assumption~\ref{eq:coefficient_condition} holds. Then there exist constants $c, C>0$ such that if $\alpha$ satisfies $\alpha^k =1$ for some positive integer $k$, then
 \[ \P \Big( P'(\alpha) = 0\Big) \le \left( \frac{C}{\lfloor \tfrac{n}{k}\rfloor}\right)^{ \tfrac{3 \deg(\ga)}{2}}  + k \exp( - c \lfloor \tfrac{n}{k}\rfloor).\]
 \end{lemma}
 The proof of the above lemma is very similar to that of Lemma~1.4 in \cite{PSZ} where the similar bound holds without the additional $k \exp( - c \lfloor \tfrac{n}{k}\rfloor)$ term for any non-constant coefficient distribution supported on $\{-1, 0, 1\}$. However, for the sake of completeness we include here a proof of Lemma~\ref{lem:roots_of_unity}. The proof of Lemma~\ref{lem:roots_of_unity} relies heavily on the following classical anti-concentration bound of S\'ark\"ozi and
Szemer\'edi \cite{SS65}.
\begin{theorem}[S\'ark\"ozi and
Szemer\'edi]\label{thm:Sarkozy_Szemeredi}
  Let $(\eps_j)_{1 \le j \le N}$ be i.i.d.\ $\mathrm{Ber}(\tfrac12)$ random variables. There exists
  a constant $C>0$ such that for any \emph{distinct} integers $(a_j)$, $1\le j\le N$, we have
  \begin{equation*}
    \max_{m\in\Z} \P\left(\sum_{j=1}^N \eps_j a_j = m\right) \le
    \frac{C}{N^{3/2}}.
  \end{equation*}
\end{theorem}

 \begin{proof}[Proof of Lemma~\ref{lem:roots_of_unity}] Since $\xi_j$ is a mixture of Bernoulli distribution, we can proceed along the same way as in the proof of Lemma~\ref{lem:unimodular_but_not_root_unity} to obtain
 \begin{equation}\label{eq:P_derivative1}
 \P \Big( P'(\alpha) = 0\Big) \le  \E  \max_{x }   \P \Big( \sum_{j =1}^n  \eps_j  j  \Delta_j \alpha^{j-1} = x \Big| (\Delta_j)_{1 \le j \le n}    \Big),
 \end{equation}
 where $(\eps_j)_{1 \le j \le n}$ are i.i.d.\  $\mathrm{Ber}(\tfrac12)$ and   $(\Delta_j)_{1 \le j \le n}$ are i.i.d.\  on $\N$.

 Observe that necessarily $\deg(\alpha)\le k$.  Set
\begin{align*}
  J&:=\{j\colon 1\le j\le n\text{ and }0\le (j -1) \bmod k\le
  \deg(\alpha) -1\},\\
  \bar{J}&:=\{1,\ldots, n\}\setminus J.
\end{align*}
Conditionally on $(\Delta_j)_{1 \le j \le n}$,  define the random variables $(S_r)$, $0\le r\le\deg(\alpha) - 1$, by
\begin{equation*}
  S_r := \sum_{
  j-1\bmod k = r} \eps_j j \Delta_j \alpha^{j-1} = \alpha^r \sum_{
  j - 1\bmod k = r} \eps_j j  \Delta_j
\end{equation*}
and
\begin{equation*}
  \bar{S} := \sum_{j\in \bar{J}} \eps_j  j \Delta_j
  \alpha^{j-1}.
\end{equation*}
Observe that
\begin{equation*}
\sum_{j =1}^n  \eps_j  j  \Delta_j \alpha^{j-1}=\sum_{j=1}^n \xi_j j \alpha^{j-1} =  \sum_{r=0}^{\deg(\alpha)-1} S_r + \bar{S}.
\end{equation*}
Now, conditionally on $(\Delta_j)_{1 \le j \le n}$,  $S_0, S_1, \ldots S_{\deg(\alpha)-1}$ and $\bar S$ are
independent. In addition,
$(\alpha^r)$, $0\le r\le \deg(\alpha)-1$, are linearly independent
over the rational numbers, and therefore the equation
$\sum_{i=0}^{\deg(\alpha)-1} a_i \alpha^i=z$
has at most one integral solution $(a_0,\ldots,a_{\deg(\alpha)-1})$
for
a given
$z\in \C$. Hence, for any given values of $(\Delta_j)_{1 \le j \le n}$ and any given $x \in \C$,
\begin{align}\label{eq:P_derivative2}
  \P\Big(\sum_{j =1}^n  \eps_j  j  \Delta_j \alpha^{j-1} = x \Big) &= \E_{\bar S} \P\left(\sum_{r=0}^{\deg(\alpha)-1} S_r= x -\bar{S}\,\Big|\,\bar{S}\right)
  \le \max_{z\in\C} \P\left(\sum_{r=0}^{\deg(\alpha)-1} S_r = z\right) \notag \\
  &= \prod_{r=0}^{\deg(\alpha)-1}\max_{z\in\C} \P(S_r = z) =
  \prod_{r=0}^{\deg(\alpha)-1}\max_{m\in\Z} \P\Bigg(\sum_{
  j-1\bmod k = r} \eps_j  j \Delta_j = m\Bigg).
\end{align}
Let $q$ be a point in the support of $\Delta_j$ and let $\eta:= \P(\Delta_j = q)>0$. Define for $0 \le r \le \deg(\alpha)-1$, $T_r:= \{ 1 \le j \le n: \Delta_j = q  \text{ and }  j-1\bmod k = r\}$. Then
\begin{equation}\label{eq:P_derivative3}
 \max_{m\in\Z} \P\Bigg(\sum_{
  j-1\bmod k = r} \eps_j  j \Delta_j = m\Bigg) \le  \max_{m'\in\mathbb{Z}} \P\Bigg(\sum_{j\in T_r} \eps_j  j = m'\Bigg) \le C |T_r|^{-3/2},
  \end{equation}
 where in the last step we apply the S\'ark\"ozi-Szemer\'edi bound (Theorem~\ref{thm:Sarkozy_Szemeredi}).  Combining \eqref{eq:P_derivative1}, \eqref{eq:P_derivative2} and \eqref{eq:P_derivative3}, we finally arrive at the inequality,
\begin{equation}\label{eq:P_derivative4}
  \P \Big( P'(\alpha) = 0\Big) \le \E \prod_{r=0}^{\deg(\alpha)-1} \min (1, C |T_r|^{-3/2}) \le \left( \frac{C}{\tfrac{\eta}{2} \lfloor \tfrac{n}{k}\rfloor}\right)^{ \tfrac{3 \deg(\ga)}{2}} +  \sum_{r=0}^{\deg(\alpha)-1} \P \Big( |T_r| \le \tfrac{\eta}{2} \lfloor \tfrac{n}{k}\rfloor \Big).
  \end{equation}
Observe that $T_r$ is Binomial random variable with  number of trials at least $ \lfloor \tfrac{n}{k}\rfloor $ and success probability $\eta$. This gives us the following bound for the left tail of $T_r$. There exists a constant $c>0$ such that  $\P \big( |T_r| \le \tfrac{\eta}{2} \lfloor \tfrac{n}{k}\rfloor \big) \le e^{-c\lfloor \tfrac{n}{k}\rfloor} $. The lemma now follows from \eqref{eq:P_derivative4} and the fact that $\deg(\ga) \le k$.
 \end{proof}
 It remains to show  \eqref{eq:low_deg_double_root_on_circle}.  Let $\ga \in \mathbb{A}$ such that $|\ga| =1, \ga \not \in \{-1, 1\}$ and $\deg(\alpha) \le 4$. First assume that $\alpha$ is a primitive  $k^{th}$ root of unity, that is, $\alpha^k = 1$ and $\alpha^l \ne 1$ for all positive integer $l < k$. Recall that $\deg(\ga) =\varphi(k)$  where $\varphi$
is Euler's totient function, i.e., $\varphi(k)=|\{1\le j\le k\colon
\gcd(j,k) = 1\}|$ (see, for example, Lemma~7.6 and Theorem~7.7 of \cite{M96}).
 By standard estimates (see \cite[Theorem 2.9]{MV07}) there exists
some constant $c_1>0$ for which
\begin{equation*}
\varphi(k) \ge \frac{c_1k}{\log\log (k+2)}.
\end{equation*}
The above bound along with the fact that  $\deg(\alpha) \le 4$ implies the bound $k \le C_2$ for some absolute constant $C_2$.  On the other hand, since $\alpha \ne \pm 1$, we have $\deg(\ga) \ge 2$. Thus, by Lemma~\ref{lem:roots_of_unity}, we deduce that $ \P(\ga \text{ is a double root  of } P) = O(n^{-3})$.

Now assume that $\alpha$ is not a root of unity. As a direct consequence of Lemma~\ref{lem:unimodular_but_not_root_unity}, we also have that $ \P(\ga \text{ is a double root  of } P) = O(n^{-5/2})$. This finishes the proof of the bound  \eqref{eq:low_deg_double_root_on_circle} and hence the proof of Lemma~\ref{lem: small deg small house}.

\section{Open Problems}
We conclude the paper with a couple of  open problems.
\begin{enumerate}

\item Define $p_{n+1}: = \max_{a \in \Z} \P \Big ( P_n(2)  = a \Big)$. Then
\begin{align*}
 p_{n+m}  = \max_{a \in \Z} \P \Big ( P_{n+m-1} (2)  = a \Big)  \ge  \max_{a \in \Z} \P \Big ( \sum_{j=0}^{m-1} \xi_j 2^j  = a \Big)   \max_{a \in \Z} \P \Big ( \sum_{j=m}^{m+n -1} \xi_j 2^j  = a \Big)  = p_mp_n.
 \end{align*}
 So, it follows from the subadditive property that there exists some $\lambda > 0$, depending on the law of $\xi_0$, such that $p_n  = e^{ - \lambda n (1+o(1))}$. However, the exact value of $\lambda$ is completely unknown.  In the special case when the maximum of atom of $\xi_0$ is at most $\tfrac12$, our Theorem~\ref{thm: main} only gives a one-sided bound $\lambda > \tfrac12 \log 2$ . It would be very interesting to investigate the dependence of value of $\lambda$ on the law of $\xi_0$ or, say, on the maximum atom of $\xi_0$.

 \item  It would be very interesting to investigate the minimum gap between the roots of a random polynomial.  Note that this problem makes sense even if the coefficient distribution is continuous. To best of our knowledge,  precise quantitive bounds on the minimum gap are not available even for the i.i.d.\  Gaussian polynomials.
 \end{enumerate}

\section*{Acknowledgments}
The authors are very grateful to Ron Peled for numerous helpful conversations and, in particular, for pointing out the reference \cite{Du}. The authors also thank Ofer Zeitouni for insightful comments.


\begin{thebibliography}{999999}



\bibitem{A78} L.V. Ahlfors, Complex analysis (third edition), McGraw-Hill Book Co., New York, 1978.

\bibitem{Ar} M. Artin, \emph{Algebra}, Prentice hall, New Jersey, 1991.

\bibitem{LM} S. Boucheron, G. Lugosi, P. Massart, \emph{A sharp concentration inequality with applications}, Random Structure and Algorithms 16 (2000), no.~3., 277--292.

\bibitem{DNV} Y. Do, H. Nguyen, and V. Vu. \emph{Real roots of random polynomials: expectation and repulsion} (2014). arXiv preprint arXiv:1409.4128.

\bibitem{D79} E. Dobrowolski, \emph{On a question of Lehmer and the number of
irreducible factors of a polynomial}, Acta Arith. {\bf 34} (1979),
no.~4, 391--401.

\bibitem{Du} A. Dubickas, \emph{On the number of polynomials of small house}, Lithuanian Mathematical Journal {\bf 39} (1999), no.~2, 168--172.

\bibitem{FK96} M. Filaseta\ and\ S. Konyagin, \emph{Squarefree values of polynomials all
of whose coefficients are $0$ and $1$}, Acta Arith. {\bf 74} (1996),
no.~3, 191--205.

\bibitem{H77}  G. Hal\'asz. \emph{Estimates for the concentration function of combinatorial number theory and probability}. Periodica Mathematica Hungarica {\bf 8}(1977),
no.~3-4, 197--211.

\bibitem{Dk} D. Kalman, \emph{A matrix proof of Newton's identities}, Math. Mag. 73 (2000), 313--315.

\bibitem{KZ} G. Kozma, O. Zeitouni, \emph{On common roots of Bernoulli polynomials}, International Mathematics Research Notices,
to appear.

\bibitem{MV07} H.L. Montgomery\ and\ R.C. Vaughan, \emph{Multiplicative number theory. {I}. Classical theory},
Cambridge Studies in Advanced Mathematics. Vol. 97, Cambridge University Press, 2007.

\bibitem{M96} P. Morandi, \emph{Field and {G}alois theory. Graduate Texts in Mathematics}. Vol.~167, Springer-Verlag, New York, 1996.

\bibitem{PSZ} R. Peled, A. Sen and O. Zeitouni, \emph{Double roots of random Littlewood polynomials}, Israel J. Math, To appear.

\bibitem{SS65} A. S\'ark\"ozi\ and\ E. Szemer\'edi, \emph{\"Uber ein Problem von Erd\H os
und Moser}, Acta Arith. {\bf 11} (1965), 205--208.

\bibitem{TV} T. Tao and V. Vu, \emph{A sharp inverse Littlewood-Offord theorem},
Random Structures \& Algorithms 37 (2010), No.~4, 525--539.

\bibitem{LL04} P. Lakatos \ and \ L.  Losonczi, \emph{Self-inversive polynomials whose zeros are on the unit circle}. Publ. Math. Debrecen 65 (2004), no~3-4, 409--420.

\bibitem{TVbook}  T. Tau \ and\ V. Vu,  \emph{Additive combinatorics,} Vol. 105, Cambridge University Press, 2006.

\bibitem{NV11} H. Nguyen and V. Vu, \emph{Optimal inverse Littlewood-Offord theorems},  Advances in Mathematics 226.6 (2011), 5298--5319.

\bibitem{NV13} H.  Nguyen \ and \ V. Vu,  \emph{Small Ball Probability, Inverse Theorems, and Applications}, Erd\"{o}s Centennial,
Bolyai Society Mathematical Studies 25, Springer Berlin, Heidelberg, 2013, pp. 409--463.


%




%
%
%
%
%
%
%
%
%
%
%
%

%
%
%
%
%
%
%
%
%
%
%
%
%
%
%
%

%

%
%

%
%
%
%
%


\end{thebibliography}
\end{document}